\documentclass[a4paper,11pt]{article}
\usepackage{amsfonts}
\usepackage{fancyhdr}
\usepackage[latin1]{inputenc}
\usepackage[T1]{fontenc}
\usepackage{rotating}
\usepackage[frenchb,english]{babel}
\usepackage{geometry}
\usepackage{graphicx}
\usepackage[colorlinks=true]{hyperref} 
\usepackage{float,caption}
\usepackage{wrapfig}
\usepackage{amsmath}
\usepackage{amsthm}
\usepackage{amssymb}
\usepackage{blkarray}
\usepackage{multirow}
\usepackage{multirow,bigdelim}
\def\VR{\kern-\arraycolsep\strut\vrule &\kern-\arraycolsep}
\def\vr{\kern-\arraycolsep & \kern-\arraycolsep}

\usepackage{tikz}
\usetikzlibrary{
  matrix,
  positioning,
  decorations,
  decorations.pathreplacing
}
\newtheorem{theorem}{Theorem}

\theoremstyle{plain}
\newtheorem*{theorem*}{Theorem}
\newtheorem{example}{Example}
\newtheorem{counter example}{Counter Example}
\newtheorem*{definition}{Definition}
\newtheorem*{remark}{Remark}
\newtheorem{proposition}{Proposition}[section]
\newtheorem{lemma}{Lemma}[section]

\newtheorem*{notation*}{Notation}

\begin{document}
\title{Moduli spaces of a family of topologically non quasi-homogeneous functions.}
\author{Jinan Loubani}
\date{September 27, 2018}

\newcommand{\Addresses}{{
  \bigskip
  \footnotesize

  \textsc{Institut de Math\'{e}matiques de Toulouse, Universit\'{e} Paul Sabatier,
    118 route de Narbonne, 31062 Toulouse cedex 9, France.}\par\nopagebreak
  \textit{E-mail address}: \texttt{Jinan.Loubani@math.univ-toulouse.fr}

  \medskip

  \textsc{Dipartimento di Matematica "F. Casorati", Universit\`{a} degli Studi di Pavia, Via Ferrata, 5 - 27100 Pavia, Italy.}\par\nopagebreak
  \textit{E-mail address}: \texttt{jinan.loubani01@universitadipavia.it}

}}

\maketitle
\begin{center}
\textbf{Abstract}
\end{center}
We consider a topological class of a germ of complex analytic function in two variables which does not belong to its jacobian ideal. Such a function is not quasi homogeneous. Each element $f$ in this class induces a germ of foliation $(df = 0)$. Proceeding similarly to the homogeneous case \cite{1} and the quasi homogeneous case \cite{2} treated by Genzmer and Paul, we describe the local moduli space of the foliations in this class and give analytic normal forms. We prove also the uniqueness of these normal forms. 

\medskip

\noindent \textbf{keywords:} complex foliations, singularities.\\
\noindent \textbf{MSC class:} 34M35, 32S65.

\section*{Introduction}
A germ of holomorphic function $f: (\mathbb{C}^2, 0) \longrightarrow (\mathbb{C}, 0)$ is said to be  \textit{quasi-homogeneous} if and only if $f$ belongs to its jacobian ideal $J(f)=( \frac{\partial f}{\partial x},\frac{\partial f}{\partial y}) $. If $f$ is quasi-homogeneous, then there exist coordinates $(x,y)$ and positive coprime integers $k$ and $l$ such that the quasi-radial vector field $R=kx\frac{\partial}{\partial x}+ly\frac{\partial}{\partial y}$ satisfies $R(f)=d\cdot f$, where the integer $d$ is the \textit{quasi-homogeneous $(k,l)$-degree} of $f$ \cite{6}. In \cite{2}, Genzmer and Paul constructed analytic normal forms of \textit{topologically quasi-homogeneous functions}, the holomorphic functions topologically equivalent to a quasi-homogeneous function.

\medskip

\noindent In this article, we study the simplest topological class beyond the quasi-homogeneous singularities, and we consider the following family of functions
$$f_{M,N}=\prod_{i=1}^N \big(y+a_ix\big)\prod_{i=1}^M \big(y+b_ix^2\big).$$
These functions are not quasi homogeneous. 
The symmetry $R$ is a central tool to study the moduli space of quasi-homogeneous functions. In some sense, it allowed Genzmer and Paul to compactify the moduli space and to describe it globally from a local study. However, in our case, we lack the existence of such a symmetry and thus we have to introduce a new approach.

\medskip

\noindent We denote by $\mathcal{T}_{M,N}$ the set of holomorphic functions which are topologically equivalent to $f_{M,N}$. The purpose of this article is to describe the moduli space $\mathcal{M}_{M,N}$ which is the topological class $\mathcal{T}_{M,N}$ up to right-left equivalence.
We give the infinitesimal description and local parametrization of this moduli space using the cohomological tools considered by J.F. Mattei in [3]: the tangent space to the moduli space is given by the first Cech cohomology group $H^1(D,\Theta_{\mathcal{F}})$, where $D$ is the exceptional divisor of the desingularization of $f_{M,N}$, and $\Theta_{\mathcal{F}}$ is the sheaf of germs of vector fields tangent to the desingularized foliation of the foliation induced by $df_{M,N} = 0$. Using a particular covering of $D$, we give a presentation of the space $H^1(D,\Theta_{\mathcal{F}})$ and exhibit a universal family of analytic normal forms. This way, we obtain local description of $\mathcal{M}_{M,N}$. We finally prove the global uniqueness of these normal forms.

\section{The dimension of $H^1(D,\Theta_{\mathcal{F}})$.}
The foliations induced by the elements of $\mathcal{T}_{M,N}$ can be desingularized after two standard blow-ups of points. So, we consider the composition of two blow-ups\\
$E: (\mathcal{M},D) \longrightarrow (\mathbb{C}^2, 0)$ with its exceptional divisor $D=E^{-1}(0)$. On the manifold $\mathcal{M}$, we consider the three charts $V_2(x_2, y_2)$, $V_3(x_3, y_3)$ and $V_4(x_4,y_4)$ in which $E$ is defined by $E(x_2, y_2)=(x_2y_2, y_2)$, $E(x_3, y_3)=(x_3, x_3^2y_3)$ and $E(x_4,y_4)=(x_4y_4,x_4y_4^2)$.
\begin{figure}[h!]
\begin{center}
\includegraphics[width=12cm]{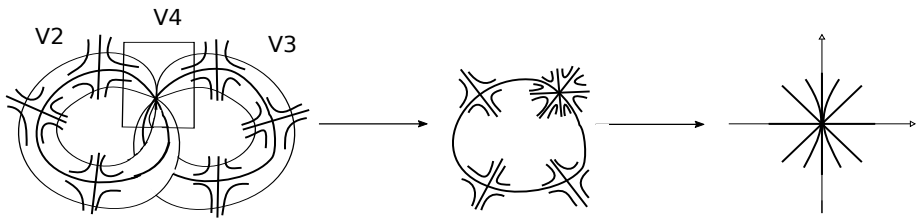}
\end{center}
\caption{Desingularization of $f_{M,N}$ for $M=N=3$}
\end{figure}

\noindent In particular, once $M\geq 2$ and $N\geq 2$, any function in $\mathcal{T}_{M,N}$ is not topologically quasi-homogeneous since the weighted desingularization process is a topological invariant \cite{7}.

\begin{notation*}
Let $\mathcal{Q}_{M,N}$ be the region in the union of the real half planes $(X,Y)$, $X\geq 0$ and $Y\geq 0$, delimited by 
$$Y-X+(M-1)>0$$  $$2Y-X-(N-1)<0$$
\end{notation*}

\begin{figure}[h!]
\begin{center}
\includegraphics[width=9cm]{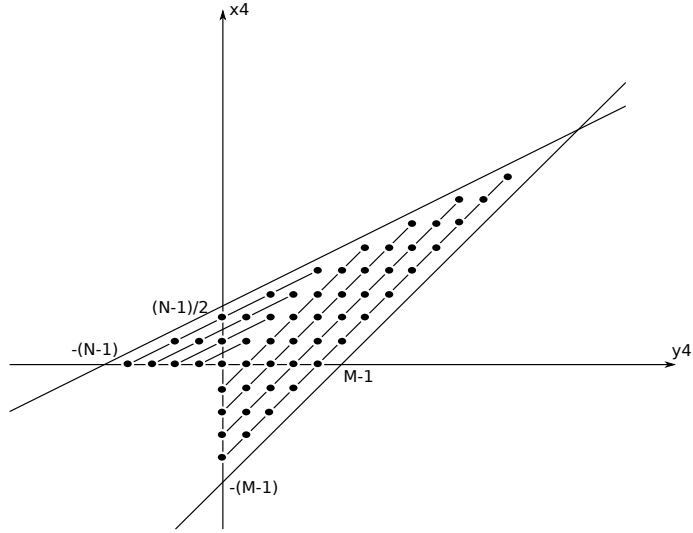}
\end{center}
\caption{The region $\mathcal{Q}_{M,N}$ for $M=N=6$}
\end{figure}

\begin{proposition}
The dimension $\delta$ of the first cohomology group $H^1(D,\Theta_{\mathcal{F}})$ is equal to the number of the integer points in the region $\mathcal{Q}_{M,N}$ which can be expressed by the following formula
$$\delta = \frac{(M+N-2)(M+N-3)}{2}+\frac{(M-1)(M-2)}{2}.$$
\end{proposition}

\begin{proof}
We consider the vector field $\theta_f$ with an isolated singularity defined by
$$\theta_f=-\frac{\partial f}{\partial x}\frac{\partial}{\partial y}+\frac{\partial f}{\partial y}\frac{\partial}{\partial x}.$$
We consider the following covering of the divisor introduced above $D=V_2\cup V_3\cup V_4.$
The sheaf $\Theta_{\mathcal{F}}$ is a coherent sheaf, and according to Siu \cite{5}, the covering $\{V_2,V_3,V_4\}$ can be supposed to be Stein. Thus, the first cohomology group $H^1(D,\Theta_{\mathcal{F}})$ is given by the quotient
$$H^1(D,\Theta_{\mathcal{F}})=\frac{H^0(V_2\cup V_4,\Theta_{\mathcal{F}})\oplus H^0(V_3\cap V_4,\Theta_{\mathcal{F}})}{\delta\big(H^0(V_2,\Theta_{\mathcal{F}})\oplus H^0(V_3,\Theta_{\mathcal{F}})\oplus H^0(V_4,\Theta_{\mathcal{F}})\big)},$$
where $\delta$ is the operator defined by $\delta(X_2,X_3,X_4)=(X_2-X_4,X_3-X_4)$.
In order to compute each term of the quotient, we consider the following vector field
$$\theta_{\textup{is}}=\frac{E^*\theta_f}{x_4^{M+N-2}y_4^{2M+N-3}}.$$
This vector field has isolated singularities and defines the foliation on the two intersections $V_2\cap V_4$ and $V_3\cap V_4$. Therefore, we have $H^0(V_2\cup V_4,\Theta_{\mathcal{F}})=\mathcal{O}(V_2\cup V_4)\cdot\theta_{\textup{is}}$ and\\ $H^0(V_3\cup V_4,\Theta_{\mathcal{F}})=\mathcal{O}(V_3\cup V_4)\cdot\theta_{\textup{is}}$, and each element $\theta_{24}$ in $H^0(V_2\cup V_4,\Theta_{\mathcal{F}})$ and $\theta_{34}$ in $H^0(V_3\cup V_4,\Theta_{\mathcal{F}})$ can be written
$$\theta_{24}=\left( \sum_{i\in \mathbb{N},j\in \mathbb{Z}}\lambda_{ij}x_4^iy_4^j\right) \cdot\theta_{\textup{is}} \textup{ and } \theta_{34}=\left( \sum_{i\in \mathbb{Z},j\in \mathbb{N}}\lambda_{ij}x_4^iy_4^j\right) \cdot\theta_{\textup{is}}.$$
Similarly, we find that the elements $\theta_{2}$ in $H^0(V_2,\Theta_{\mathcal{F}})$ and $\theta_{3}$ in $H^0(V_3,\Theta_{\mathcal{F}})$ can be written
$$\theta_{2}=\left( \sum_{i,j\in \mathbb{N}}\alpha_{ij}x_4^jy_4^{2j-i-(N-1)}\right) \cdot\theta_{\textup{is}} \textup{ and } \theta_{3}=\left( \sum_{i,j\in \mathbb{N}}\beta_{ij}x_4^{i-j-(M-1)}y_4^i\right) \cdot\theta_{\textup{is}}.$$
The cohomological equation describing $H^1(D,\Theta_{\mathcal{F}})$ is thus equivalent to 
$$\left\lbrace \begin{array}{rl}
\theta_{24} & =\theta_2-\theta_4 \\
  0 & =\theta_3-\theta_4  \\
\end{array}\right. \Longleftrightarrow \theta_{24}=\theta_2-\theta_3
\textup{ and }
\left\lbrace \begin{array}{rl}
0 & =\theta_2-\theta_4  \\
  \theta_{34} & =\theta_3-\theta_4 
\end{array}\right. \Longleftrightarrow \theta_{34}=\theta_3-\theta_2,$$
which means that its dimension corresponds to the
number of elements which do not have a solution in any of the above two systems. 
This implies that the dimension of the cohomology group is equal to the number of integer points in the region $\mathcal{Q}_{M,N}$ that can be expressed by the following formula
$$\delta = \frac{(M+N-2)(M+N-3)}{2}+\frac{(M-1)(M-2)}{2}.$$
\end{proof}

\section{The local normal forms.}
We denote by $\mathcal{P}$ the following open set of $\mathbb{C}^{\delta}$
\begin{multline*}
\mathcal{P} = \left\lbrace (\cdots,a_{k,i},\cdots,b_{k',i'},\cdots) \textup{ such that }  a_{1,i}\neq 0, b_{1,j}\neq 0,1 \textup{ and }  \right. \\
\left. a_{1,i}\neq a_{1,j}, b_{1,i'}\neq b_{1,j'} \textup{ for } i\neq j \textup{ and } i'\neq j'\right\rbrace,
\end{multline*}
where the indexes $k$,$i$,$k'$ and $i'$ satisfy the following system of inequalities
\[\left\lbrace 
\begin{array}{rcl}
0 \leq & k-1 & \leq i-1\\
-(N-2) \leq & 2k-i-1 & \leq 2i-2\\
-(M-2) \leq & k'-i'-1 & \leq  N-3+2i'\\
0 \leq & k'-1 & \leq N-2+2i'
\end{array}
\right.  \]
For $p\in\mathcal{P}$, we define the analytic normal form by 
$$N_p^{(M,N)}=xy(y+x^2)\prod_{i=1}^{N-1}\left(y+\sum_{k=1}^{i}a_{k,i}xy^{k-1}\right)\prod_{i=1}^{M-2}\left(y+\sum_{k=1}^{N-1+2i}b_{k,i}x^{k+1}\right).$$
We consider the saturated foliation $\mathcal{F}_p^{(M,N)}$ defined by the one-form $dN_p^{(M,N)}$ on $\mathbb{C}^{2+\delta}$.

\medskip

\noindent The main result of this article is the following:
\begin{theorem}\label{main}
For any $p_0$ in $\mathcal{P}$ the germ of unfolding $\left\lbrace \mathcal{F}_p^{(M,N)}, p\in (\mathcal{P},p_0)\right\rbrace $ is a universal equireducible unfolding of the foliation $\mathcal{F}_{p_0}^{(M,N)}$.
\end{theorem}

\noindent In particular, for any equireducible unfolding $\mathcal{F}_t$, $t \in (\mathcal{T}, t_0)$ which defines $\mathcal{F}_{p_0}^{(M.N)}$ for $t=t_0$, there exists a map $\lambda : (\mathcal{T},t_0) \longrightarrow (\mathcal{P},p_0)$ such that the family $\mathcal{F}_t$ is analytically equivalent to $N_{\lambda(t)}$. Furthermore, the differential of  $\lambda$ at the point $t_0$ is unique. As for the uniqueness of the map $\lambda$, it follows from Theorem \ref{main2}.

\medskip

\noindent Consider the sheaf $\Theta_{\mathcal{F}_{p_0}^{(M,N)}}$ of germs of vector fields tangent to the desingularized foliation $\tilde{\mathcal{F}}_{p_0}^{(M,N)}$ of the foliation $\mathcal{F}_{p_0}^{(M,N)}$ induced by $dN^{(M,N)}_{p_0} = 0$. According to \cite{3}, one can define the derivative of the deformation as a map from $T_{p_0}\mathcal{P}$ into $H^1\left( D,\Theta_{\mathcal{F}^{(M,N)}_{p_0}}\right) $. We denote this map by $T\mathcal{F}_{p_0}^{(M,N)}$: since (see \cite{3}) after desingularization any equireducible unfolding is locally analytically trivial, there exists $X_l$, $l\in\{2,3,4\}$, a collection of local vector fields solutions of 
\begin{equation}
\frac{\partial\tilde{N}_p^{(M,N)}}{\partial p_{1,i}}=\alpha_{1,i}(x_l,y_l,a_{1,i},b_{1,i})\frac{\partial\tilde{N}_p^{(M,N)}}{\partial x_l}+\beta_{1,i}(x_l,y_l,a_{1,i},b_{1,i})\frac{\partial\tilde{N}_p^{(M,N)}}{\partial y_l},
\end{equation}
where $p_{1,i}\in\{a_{1,i},b_{1,i}\}$. The cocycle $\{X_{2,4}=X_{2}-X_{4},X_{3,4}=X_{3}-X_{4}\}$ evaluated at $p=p_0$ is the image of the direction $\frac{\partial}{\partial p_{1,i}}$ in $H^1\left( D,\Theta_{\mathcal{F}^{(M,N)}_{p_0}}\right) $ by $T\mathcal{F}_{p_0}^{(M,N)}$. To prove Theorem \ref{main}, we will make use of the following result:

\begin{theorem*}[\cite{3}]
The unfolding $\mathcal{F}_t$, $t \in (\mathcal{T}, t_0)$ is universal among the equireducible unfoldings of $\mathcal{F}_{t_0}$ if and only if the map $T\mathcal{F}_{t_0}: T_{t_0}\mathcal{T} \longrightarrow H^1(D,\Theta_{\mathcal{F}})$ is a bijective map.
\end{theorem*}

\noindent Theorem \ref{main} is thus a consequence of the following proposition.

\begin{proposition}
We consider the unfolding $\tilde{\mathcal{F}}_p^{(M,N)}$ defined by the blowing up of $N_p^{(M,N)}$, $p\in (\mathcal{P},p_0)$. The image of the family $\left\lbrace \frac{\partial}{\partial a_{k,i}},\frac{\partial}{\partial b_{k,i}}\right\rbrace _{k,i}$ in $H^1\left( D,\Theta_{\mathcal{F}^{(M,N)}_{p_0}}\right) $ by $T\mathcal{F}_{p_0}^{(M,N)}$ is linearly free.
\end{proposition}

\noindent Let $S$ be the subset of $\mathcal{P}$ defined by its elements at the first level $k=k'=1$ i.e.
$$S= \left\lbrace (\cdots,a_{1,i},\cdots,b_{1,i'},\cdots) \textup{ such that } 1\leq i\leq N-1 \textup{ and }  1\leq i'\leq M-2\right\rbrace .$$
We denote by $A_1$ the square matrix of size $M+N-3$, representing the decomposition of the images of $\left\lbrace \frac{\partial}{\partial a_{1,i}},\frac{\partial}{\partial b_{1,i}}\right\rbrace $ in $H^1\left( D,\Theta_{\mathcal{F}^{(M,N)}_{p_0}}\right) $ by $T\mathcal{F}_{p_0}^{(M,N)}$ on the corresponding basis. We note that the corresponding basis is in bijection with the set 
$$\left\lbrace x^{\alpha}y^{\beta}/(\alpha,\beta)=(0,1-i),1\leq i\leq N-1 \textup{ or } (\alpha,\beta)=(-i,0), 1\leq i\leq M-2\right\rbrace .$$
Therefore, the proof of the proposition results from the following two lemmas.

\begin{lemma}
The matrix $A_1$ is invertible.
\end{lemma}

\begin{proof}
The matrix $A_1$ is given by
$$A_1=\bordermatrix{ & \frac{\partial}{\partial a_{1,1}} & \frac{\partial}{\partial a_{1,2}} & \ldots & \frac{\partial}{\partial a_{1,N-1}} & & \frac{\partial}{\partial b_{1,1}} & \frac{\partial}{\partial b_{1,2}} & \ldots & \frac{\partial}{\partial b_{1,M-2}}\cr
\frac{1}{y_4^{N-2}} &  &  &  &  & \vline &  &  &  &   \cr
\frac{1}{y_4^{N-3}} &  &  &  &  & \vline & &  &  &   \cr
\vdots &  &  & M_1 &  &  \vline & &   M_2 &  &   \cr
\frac{1}{y_4} &  &  &  &  & \vline &  &  &  &   \cr
1 &  &  &  &  & \vline &  &  &  &   \cr
 \hline  \cr
\frac{1}{x_4} &  &  &  &  & \vline &  &  &  &   \cr
\frac{1}{x_4^2} &  &  &  &  &  \vline &  &  &  &   \cr
\vdots &  &  & M_3 &  & \vline &  & M_4 &  &   \cr
\frac{1}{x_4^{M-2}} &  &  &  &  & \vline &  &  &  &   \cr
}.$$

\noindent We start by computing the matrix $M_3$. In the chart $V_4$, we have to solve
\begin{equation}
\frac{\partial\tilde{N}_p^{(M,N)}}{\partial a_{1,i}}=\alpha_{1,i}(x_4,y_4,a_{1,i},b_{1,i})\frac{\partial\tilde{N}_p^{(M,N)}}{\partial x_4}+\beta_{1,i}(x_4,y_4,a_{1,i},b_{1,i})\frac{\partial\tilde{N}_p^{(M,N)}}{\partial y_4}.
\end{equation}
Since $E$ is defined on $V_4$ by $E(x_4,y_4)=(x_4y_4,x_4y_4^2)$, we find that 
\begin{multline*}
\tilde{N}_p^{(M,N)}(x_4,y_4)=x_4^{M+N}y_4^{2M+N}(1+x_4)\\
\prod_{i=1}^{N-1}\left(y_4+\sum_{k=1}^{i}a_{k,i}x_4^{k-1}y_4^{2k-2}\right)\prod_{i=1}^{M-2}\left(1+\sum_{k=1}^{N-1+2i}b_{k,i}x_4^{k}y_4^{k-1}\right).
\end{multline*}
We have
$$\frac{\partial\tilde{N}_p^{(M,N)}}{\partial a_{1,i}} = \frac{\tilde{N}_p^{(M,N)}}{y_4+\sum_{k=1}^{i}a_{k,i}x_4^{k-1}y_4^{2k-2}}= \frac{y_4^{2M+N}}{a_{1,i}}\left( Q(x_4)+y_4(...)\right) $$
with 
$$Q(x_4)=x_4^{M+N}(1+x_4)\prod_{j=1}^{N-1}a_{1,j}\prod_{j=1}^{M-2}(1+b_{1,j}x_4)$$
and where the suspension points (...) correspond to auxiliary holomorphic functions in $(x_4,y_4)$.
Since $\tilde{N}_p^{(M,N)}=y_4^{2M+N}\left( Q(x_4)+y_4(...)\right) $, we find that 
\begin{equation}
\begin{array}{rcl}
\frac{\partial\tilde{N}_p^{(M,N)}}{\partial x_4}&=&y_4^{2M+N}\left( Q'(x_4)+y_4(...)\right) \\
\frac{\partial\tilde{N}_p^{(M,N)}}{\partial y_4}&=&(2M+N)y_4^{2M+N-1}Q(x_4)+y_4^{2M+N}(...) 
\end{array}
\end{equation}
Setting $\beta_{1,i}=y_4\tilde{\beta}_{1,i}$, we deduce from $(2)$ that 
\begin{equation}
\frac{Q(x_4)}{a_{1,i}}=\alpha_{1,i}(x_4,0)Q'(x_4)+(2M+N)\tilde{\beta}_{1,i}(x_4,0)Q(x_4)+y_4(...) 
\end{equation}
Using B\'{e}zout identity, there exist polynomials $W$ and $Z$ in $x_4$ such that
$$Q\wedge Q' = WQ' + ZQ$$
where $Q\wedge Q'$ is the great common divisor of $Q$ and $Q' $. We can choose the polynomial function $W$ to be of degree $M-1$. We denote by 
$$S(x_4)=x_4(1+x_4)\prod_{i=1}^{M-2}(1+b_{1,i}x_4)$$
the polynomial function satisfying $Q=(Q\wedge Q')S$. Therefore we obtain a solution of $(2)$ in the chart $V_4$ of the form
$$\begin{array}{rcl}
\alpha_{1,i} & = & \frac{W(x_4)S(x_4)}{a_{1,i}}+y_4(...)\\
\beta_{1,i} & = & \frac{y_4}{2M+N}\frac{Z(x_4)S(x_4)}{a_{1,i}}+y_4^2(...)\\
\textup{i.e. } X_{1,i}^{(4)} & = & \frac{W(x_4)S(x_4)}{a_{1,i}}\frac{\partial}{\partial x_4}+y_4(...).
\end{array}$$
Similarly, in the chart $V_3$ we write
$$\tilde{N}_p^{(M,N)}=x_3^{2M+N}(P(y_3)+x_3(...))$$
with
$$P(y_3)=y_3(y_3+1)\prod_{j=1}^{N-1}a_{1,j}\prod_{j=1}^{M-2}(y_3+b_{1,j}).$$
We set $P\wedge P' = UP' + VP$ and $P=(P\wedge P')R$ with 
$$R=y_3(y_3+1)\prod_{i=1}^{M-2}(y_3+b_{1,i}).$$ 
Also, we can assume that the degree of $U$ is $M-1$ and so we obtain the solution
$$X_{1,i}^{(3)}=\frac{U(y_3)R(y_3)}{a_{1,i}}\frac{\partial}{\partial y_3}+x_3(...).$$
To compute the cocycle we write $X_{1,i}^{(3)}$ in the chart $V_4$. Using the standard change of coordinates $x_4=1/y_3$ and $y_4=x_3y_3$ and since we have
$$U(y_3)=\frac{\tilde{U}(x_4)}{x_4^{M-1}} \textup{ and } R(y_3)=\frac{S(x_4)}{x_4^{M+1}}$$
where $\tilde{U}$ is a polynomial function, we find the first part of the first term of the cocycle
$$X_{1,i}^{(3,4)}=X_{1,i}^{(3)}-X_{1,i}^{(4)}=-\frac{S(x_4)}{a_{1,i}}\left[ \frac{\tilde{U}(x_4)}{x_4^{2M-2}}+W(x_4)\right] \frac{\partial}{\partial x_4}+y_4(...).$$
Let $\Theta_0$ be a holomorphic vector field with isolated singularities defining $\tilde{\mathcal{F}}_{p_0}^{(M,N)}$ on $V_3 \cap V_4$. We have
$$X_{1,i}^{(3,4)}=\Phi_{1,i}^{(3,4)}\Theta_0.$$
We can choose $\Theta_0=\frac{E^*\Theta_{N_p^{(M,N)}}}{x_4^{M+N-2}y_4^{2M+N-3}}$ with $\Theta_{N_p^{(M,N)}}=\frac{\partial N_p^{(M,N)}}{\partial x}\frac{\partial}{\partial y}-\frac{\partial N_p^{(M,N)}}{\partial y}\frac{\partial}{\partial x}$. According to Proposition $(1.1)$, the set of the coefficients of the Laurent's series of $\Phi_{1,i}^{(3,4)}$ characterizes the class of $X_{1,i}^{(3,4)}$ in $H^1(D,\Theta_{\mathcal{F}^{(M,N)}_{p_0}})$. Now, according to $(3)$, we get the equality
$$\Phi_{1,i}^{(3,4)}=\frac{1}{(2M+N)a_{1,i}\prod_{j=1}^{N-1}a_{1,j}}\left[ \frac{\tilde{U}(x_4)}{x_4^{2M-2}}+W(x_4)\right] +y_4(...).$$
Since $\tilde{U}(x_4)$ is of degree $M-1$, then the coefficients of $1/x_4^l$ for $1\leq l\leq M-2$ in the Laurent series of  $\frac{\tilde{U}(x_4)}{x_4^{2M-2}}$ are zeros. So the matrix $M_3$ is the zero matrix.

\medskip

\noindent We proceed similarly to compute the matrix $M_4$. So, in the chart $V_4$, we have to solve the following equation
\begin{equation}
\frac{\partial\tilde{N}_p^{(M,N)}}{\partial b_{1,i}}=\eta_{1,i}(x_4,y_4,a_{1,i},b_{1,i})\frac{\partial\tilde{N}_p^{(M,N)}}{\partial x_4}+\gamma_{1,i}(x_4,y_4,a_{1,i},b_{1,i})\frac{\partial\tilde{N}_p^{(M,N)}}{\partial y_4}.
\end{equation}
Following the same algorithm, we obtain the second part of the first term of the cocycle
$$Y_{1,i}^{(3,4)}=Y_{1,i}^{(3)}-Y_{1,i}^{(4)}=-\frac{S(x_4)}{1+b_{1,i}x_4}\left[ \frac{\tilde{U}(x_4)}{x_4^{2M-3}}+x_4W(x_4)\right] \frac{\partial}{\partial x_4}+y_4(...).$$
Setting $Y_{1,i}^{3,4}=\Psi_{1,i}^{3,4}\Theta_0$, we obtain the following expression of $\Psi_{1,i}^{(3,4)}$
$$\Psi_{1,i}^{(3,4)}=\frac{1}{(2M+N)\prod_{j=1}^{N-1}a_{1,j}(1+b_{1,i}x_4)}\left[ \frac{\tilde{U}(x_4)}{x_4^{2M-3}}+x_4W(x_4)\right] +y_4(...).$$
Now, to study the invertibility of the matrix $M_4$, we write
$$\tilde{U}(x_4)=\sum_{l=0}^{M-1}u_lx_4^l \textup{ and } \frac{1}{1+b_{1,i}x_4}=\sum_{s=0}^{\infty}(-1)^{s}b_{1,i}^sx_4^s.$$
So, we obtain the following equality
$$\frac{\tilde{U}(x_4)}{(1+b_{1,i})x_4^{2M-3}}=\sum_{j=1}^{M-2}d_{ji}\frac{1}{x_4^{M-j-1}}+\frac{T(x_4)}{x_4^{2M-3}}+x_4(...)+\textup{cst},$$
where $T$ is a polynomial in $x_4$ of degree $M-2$ and $d_{ji}$ is given by 
$$d_{ji}=\sum_{r=0}^{M-1}(-1)^{M-r+j}u_rb_{1,i}^{M+j-r-2}=(-1)^{M+j}b_{1,i}^{M+j-2}\tilde{U}\left(\frac{-1}{b_{1,i}}\right).$$
This yields the following expression of $\Psi_{1,i}^{(3,4)}$
\begin{multline*}
\Psi_{1,i}^{(3,4)}=\frac{1}{(2M+N)\prod_{l=1}^{N-1}a_{1,l}}\\
\left[ \sum_{j=1}^{M-2}\frac{(-1)^{j+1}b_{1,i}^{2M-j-3}}{x_4^j}\tilde{U}\left( \frac{-1}{b_{1,i}}\right) +\frac{T(x_4)}{x_4^{2M-3}}+x_4(...)+\textup{cst}\right] +y_4(...).
\end{multline*}
Thus, the matrix $M_4=(m_{ji})_{1\leq i,j\leq M-2}$ is given by 
$$m_{ji}=\frac{(-1)^{j+1}b_{1,i}^{2M-j-3}}{(2M+N)\prod_{l=1}^{N-1}a_{1,l}}\tilde{U}\left( \frac{-1}{b_{1,i}}\right) \textup{ } \forall 1\leq i,j\leq M-2$$
which defines a Vandermonde matrix.
We note that $\tilde{U}\left( \frac{-1}{b_{1,i}}\right)$ is different from zero for all $1\leq i\leq M-2$ because the different values $\{-b_{1,i}\}_{1\leq i\leq M-2}$ are roots of the polynomial $P$ which satisfies the B\'{e}zout identity $P\wedge P'=UP' +VP$. So the matrix $M_4$ is invertible.

\medskip

\noindent Now we compute the second cocycle.
In the chart $V_4$, we can write $\tilde{N}_P^{(M,N)}$ as 
$$\tilde{N}_P^{(M,N)}=x_4^{M+N}\left( A(y_4)+y_4^{2M+N}x_4(...)\right) $$
where $A(y_4)=y_4^{2M+N}\prod_{j=1}^{N-1}(y_4+a_{1,j}).$
So, we obtain the following expressions
\begin{equation}
\begin{array}{rcl}
\frac{\partial\tilde{N}_p^{(M,N)}}{\partial a_{1,i}} & = & \frac{x_4^{M+N}}{y_4+a_{1,i}}\left( A(y_4)+y_4^{2M+N}x_4(...)\right)\\
\frac{\partial\tilde{N}_p^{(M,N)}}{\partial x_4} & = & (M+N)x_4^{M+N-1}A(y_4)+y_4^{2M+N}x_4^{M+N}(...)\\
\frac{\partial\tilde{N}_p^{(M,N)}}{\partial y_4} & = & x_4^{M+N}\left( A'(y_4)+y_4^{2M+N-1}x_4(...)\right) 
\end{array}
\end{equation}
Setting $\alpha_{1,i}=x_4\tilde{\alpha}_{1,i}$, we deduce from $(2)$ that 
\begin{equation}
\frac{A(y_4)}{y_4+a_{1,i}}=(M+N)\tilde{\alpha}_{1,i}(0,y_4)A(y_4)+\beta_{1,i}(0,y_4)A'(y_4)+y_4^{2M+N-1}x_4(...).
\end{equation}
Using B\'{e}zout identity, there exist polynomials $B$ and $C$ in $y_4$ such that
$$A\wedge A' = BA' + CA.$$
As before, we can choose the polynomial function $B$ to be of degree $N-1$. We denote by $D(y_4)=y_4\prod_{j=1}^{N-1}(y_4+a_{1,j})$ the polynomial function satisfying $A=(A\wedge A')D$. Therefore we obtain a solution of $(2)$ in the chart $V_4$
$$X_{1,i}^{(4)}= \frac{B(y_4)D(y_4)}{y_4+a_{1,i}}\frac{\partial}{\partial y_4}+x_4(...).$$
Similarly, in the chart $V_2$ we write
$$\tilde{N}_p^{(M,N)}=y_2^{M+N}(J(x_2)+x_2^2y_2(...))$$
with
$$J(x_2)=x_2\prod_{j=1}^{N-1}(1+a_{1,j}x_2).$$
We set $J\wedge J' = KJ' + LJ=1$. Again, we can assume that the degree of $K$ is $N-1$ and so we obtain the solution
$$X_{1,i}^{(2)}=\frac{x_2}{1+a_{1,i}x_2}K(x_2)J(x_2)\frac{\partial}{\partial x_2}+y_2(...).$$
Using the change of coordinates $x_4=x_2^2y_2$ and $y_4=1/x_2$, we find the first part of the second term of the cocycle
$$X_{1,i}^{(2,4)}=X_{1,i}^{(2)}-X_{1,i}^{(4)}=-\frac{1}{y_4+a_{1,i}}\left[ \frac{\tilde{K}(y_4)A(y_4)}{y_4^{2M+3N-3}}+B(y_4)D(y_4)\right] \frac{\partial}{\partial y_4}+x_4(...)$$
where $\tilde{K}$ is the polynomial function satisfying $K(x_2)=\frac{\tilde{K}(y_4)}{y_4^{N-1}}$.\\
Finally, we obtain the following expression of $\Phi_{1,i}^{(2,4)}$
$$\Phi_{1,i}^{(2,4)}=\frac{-1}{(M+N)(y_4+a_{1,i})}\left[ \frac{\tilde{K}(y_4)}{y_4^{2N-2}}+B(y_4)\right] +x_4(...).$$
Similarly, we find that $\Phi_{1,i}^{(2,4)}$ can be written as 
$$\Phi_{1,i}^{(2,4)}=\frac{-1}{M+N}\left[ \sum_{j=1}^{N-1}\frac{(-1)^{N+j-1}\tilde{K}(-a_{1,i})}{a_{1,i}^{N+j}}\frac{1}{y_4^{N-j-1}}+\frac{B(0)}{a_{1,i}}+\frac{R(y_4)}{y_4^{2N-2}}+y_4(...)\right] +x_4(...).$$
So, the matrix $M_1=(m_{ji})_{1\leq i,j \leq N-1}$ is given by 
$$m_{ji}=\left\lbrace \begin{array}{ll}
\frac{(-1)^{N+j}}{(M+N)a_{1,i}^{N+j}}\tilde{K}(-a_{1,i}) & \textup{ for } j\neq N-1\\
\frac{1}{M+N}\left( \frac{-1}{a_{1,i}^{2N-1}}\tilde{K}(-a_{1,i})-\frac{B(0)}{a_{1,i}}\right)  & \textup{ for } j=N-1. \end{array} \right.$$
A simple computation shows that the determinant of the matrix $M_1$ is given by 
\begin{multline*}
\textup{det}(M_1)=\frac{(-1)^{N^2-1}}{(M+N)^{N-1}}\prod_{i=1}^{N-1}\frac{\tilde{K}(-a_{1,i})}{a_{1,i}^{N+1}}\\
\left[ \prod_{1\leq i<j\leq N-1}\left( \frac{1}{a_{1,i}}-\frac{1}{a_{1,j}}\right)- B(0)\sum_{i=1}^{N-1}\frac{(-1)^{i}a_{1,i}^N}{\tilde{K}(-a_{1,i})}\mathcal{M}_{(N-1)i}\right] 
\end{multline*}
where $\mathcal{M}_{(N-1)i}=\prod_{\substack{1\leq j<j'\leq N-1\\ j,j'\neq i}}\left( \frac{1}{a_{1,j}}-\frac{1}{a_{1,j'}}\right)$ is the determinant of the matrix obtained by deleting the $(N-1)^{\textup{th}}$ row and $i^{\textup{th}}$ column of the Vandermonde $(N-1)$-matrix of $\left\lbrace \frac{-1}{a_{1,i}}\right\rbrace _{1\leq i\leq N-1}$. 
\medskip

\noindent Let us compute the term $B(0)\sum_{i=1}^{N-1}\frac{(-1)^{i}a_{1,i}^N}{\tilde{K}(-a_{1,i})}\mathcal{M}_{(N-1)i}$. In fact, we know that \\$\tilde{K}(y_4)=y_4^{N-1}K(x_2)$ with $y_4=1/x_2$. This implies that 
$$\tilde{K}(-a_{1,i})=(-a_{1,i})^{N-1}K\left( -\frac{1}{a_{1,i}}\right) .$$
But, we also know that $K\left( -\frac{1}{a_{1,i}}\right) =\frac{1}{J'\left( \frac{-1}{a_{1,i}}\right) }$.
Computing the term $J'\left(\frac{-1}{a_{1,i}} \right) $, we get the following expression
$$\tilde{K}(-a_{1,i})=\frac{(-1)^{N}a_{1,i}^{N-1}}{\prod_{\substack{j=1\\j\neq i}}^{N-1}a_{1,j}\left( \frac{1}{a_{1,j}}-\frac{1}{a_{1,i}}\right) }.$$
Moreover, one can see that the term $(-1)^i\prod_{\substack{j=1\\j\neq i}}^{N-1}\left( \frac{1}{a_{1,j}}-\frac{1}{a_{1,i}}\right) \mathcal{M}_{(N-1)i}$ is equal to \\$(-1)^{\alpha +i}\prod_{\substack{1\leq i<j\leq N-1}}\left( \frac{1}{a_{1,i}}-\frac{1}{a_{1,j}}\right) $, where $\alpha$ is equal to the number of integer numbers in the interval $[i+1,N-1]$. When $N$ is even $(-1)^{\alpha +i}$ is equal to $-1$ but when $N$ is odd it is equal to $1$. This implies that we have the following equality
$$B(0)\sum_{i=1}^{N-1}\frac{(-1)^{i}a_{1,i}^N}{\tilde{K}(-a_{1,i})}\mathcal{M}_{(N-1)i}=-(N-1)B(0)\prod_{j=1}^{N-1}a_{1,j}\prod_{1\leq i<j\leq N-1}\left( \frac{1}{a_{1,i}}-\frac{1}{a_{1,j}}\right) .$$
A simple computation using B\'{e}zout identity shows that the term $B(0)$ is given by
$$B(0)=\frac{1}{(2M+N)\prod_{j=1}^{N-1}a_{1,j}}.$$
Finally, we get the following expression of the determinant of the matrix $M_1$
$$\textup{det}(M_1)=\frac{(-1)^{N^2-1}}{(M+N)^{N-1}}\frac{2M+2N-1}{2M+N}\prod_{i=1}^{N-1}\frac{\tilde{K}(-a_{1,i})}{a_{1,i}^{N+1}}\prod_{1\leq i<j\leq N-1}\left( \frac{1}{a_{1,i}}-\frac{1}{a_{1,j}}\right) .$$
Like for $\tilde{U}$, we also have that $\tilde{K}(-a_{1,i})$ is different from zero for all $1\leq i\leq N-1$ and $a_{1,i}$ is different from $a_{1,j}$ for all $i\neq j$. This ensures that the matrix $M_1$ is invertible.
\end{proof}

\begin{lemma}
The square matrix $\mathcal{A}$ of size $\delta$, representing the decomposition of the images of $\{\frac{\partial}{\partial a_{k,i}},\frac{\partial}{\partial b_{k,i}}\}_{k,i}$ in $H^1(D,\Theta_{\mathcal{F}^{(M,N)}_{p_0}})$ by $T\tilde{\mathcal{F}}_p(p_0)$ on its basis, is an invertible matrix.
\end{lemma}

\begin{proof}
After proving the invertibility of the matrix $A_1$, it remains to study the propagation of these coefficients along the higher levels. In fact, we have to solve the following equations
\begin{equation}
\frac{\partial\tilde{N}_p^{(M,N)}}{\partial a_{k,i}}=\alpha_{k,i}(x_4,y_4,a_{k,i},b_{k,i})\frac{\partial\tilde{N}_p^{(M,N)}}{\partial x_4}+\beta_{k,i}(x_4,y_4,a_{k,i},b_{k,i})\frac{\partial\tilde{N}_p^{(M,N)}}{\partial y_4}
\end{equation}
\begin{equation}
\frac{\partial\tilde{N}_p^{(M,N)}}{\partial b_{k,i}}=\eta_{k,i}(x_4,y_4,a_{k,i},b_{k,i})\frac{\partial\tilde{N}_p^{(M,N)}}{\partial x_4}+\gamma_{k,i}(x_4,y_4,a_{k,i},b_{k,i})\frac{\partial\tilde{N}_p^{(M,N)}}{\partial y_4}.
\end{equation}
We note that we have the following relations
\begin{equation}
\frac{\partial\tilde{N}_p^{(M,N)}}{\partial a_{k,i}}=x_4^{k-1}y_4^{2k-2}\frac{\partial\tilde{N}_p^{(M,N)}}{\partial a_{1,i}} \textup{ and } \frac{\partial\tilde{N}_p^{(M,N)}}{\partial b_{k,i}}=x_4^{k-1}y_4^{k-1}\frac{\partial\tilde{N}_p^{(M,N)}}{\partial b_{1,i}}.
\end{equation}
This implies that if $X_{k,i}=\alpha_{k,i}\frac{\partial}{\partial x_4}+\beta_{k,i}\frac{\partial}{\partial y_4}$ and $Y_{k,i}=\eta_{k,i}\frac{\partial}{\partial x_4}+\gamma_{k,i}\frac{\partial}{\partial y_4}$ are solutions of $(8)$ and $(9)$ respectively for $k=1$, then we obtain solutions for the other values of $k$ setting
$$X_{k,i}=x_4^{k-1}y_4^{2k-2}X_{1,i} \textup{ and } Y_{k,i}=x_4^{k-1}y_4^{k-1}Y_{1,i}.$$
This propagation can be described using the region $\mathcal{Q}_{M,N}$ as shown in figure $(2)$.
In fact, the decomposition of the vector fields $X_{k,i}^{(2,4)}$, $X_{k,i}^{(3,4)}$, $Y_{k,i}^{(2,4)}$ and $Y_{k,i}^{(3,4)}$ on the basis of $H^1\left( D,\Theta_{\mathcal{F}^{(M,N)}_{p_0}}\right) $ corresponds to the decomposition of the series $\Phi_{k,i}^{(2,4)}$, $\Phi_{k,i}^{(3,4)}$, $\Psi_{k,i}^{(2,4)}$ and $\Psi_{k,i}^{(3,4)}$ on the basis
$$\left\lbrace x_4^iy_4^j\mid (i,j)\in \mathbb{N}\times\mathbb{Z}\cup\mathbb{Z}\times\mathbb{N} \textup{ such that } j-2i+(N-1)>0 \textup{ and } j-i-(M-1)<0\right\rbrace .$$ 
As a consequence of the previous relations, this decomposition can be expressed by the following matrix
$$\mathcal{A}=\begin{bmatrix} 
 A_1 & 0 & 0 & \cdots & 0 \\
 * & A_2 & 0 & \cdots & 0 \\
 * & * & A_3 & \cdots & 0 \\
 \vdots & \vdots & \vdots & \ddots & \vdots \\
 * & *  & * & * & A_{N+2M-5}
 \end{bmatrix}
$$
where 
$A_1=\begin{bmatrix}
M_1 & M_2\\
M_3 & M_4
\end{bmatrix}$ 
and $A_k$ is given by 

$$\bordermatrix{ & \frac{\partial}{\partial a_{k,1}} &  \ldots & \frac{\partial}{\partial a_{k,N-k}} & & \frac{\partial}{\partial b_{k,1}} &  \ldots & \frac{\partial}{\partial b_{k,M-2}}\cr
\frac{x_4^{k-1}}{y_4^{N-2k}}   &  & M_1^k=M_1\setminus \textup{last} &  & \vline &  &  &  &   \cr
\vdots   &  & k-1 &  &  \vline & &   0 &  &   \cr
x_4^{k-1}y_4^{k-1}   &  & \textup{column and row} &  & \vline &  &  &  &   \cr
 \hline   \cr
x_4^{k-2}y_4^{k-1}   &  &  &  & \vline &  &   &  &   \cr
\vdots   &  & 0 &  & \vline &  & M_4  &  &   \cr
\frac{y_4^{k-1}}{x_4^{M-k-1}}   &  &  &  & \vline &  &  &  &   \cr
} \textup{ if } 2\leq k\leq N-1 $$

$$\bordermatrix{ & \frac{\partial}{\partial b_{k,M-1-q_{k}}} &  \ldots & \frac{\partial}{\partial b_{k,M-2}}\cr
 \frac{y_4^{k-1}}{x_4^{M-k-q_{k}}} & & M_4^k=M_4\setminus \textup{first} &  \cr
 & &  M-2-q_{k}  &    \cr
\frac{y_4^{k-1}}{x_4^{M-k-1}}  & & \textup{column and row}   &   \cr
} \hspace{0.5cm} \textup{ if } N\leq k\leq N+2M-5 $$
with $q_{k}=]\frac{k-1+(N-1)}{2}]+M-k$, where $]x]$ is the strict integer part $m$ of $x$ defined by\\ $m < x \leq m+ 1$. For $2\leq k\leq N-1$, the determinant of the matrix $M_1^k$ is given by 
$$\textup{Vandermonde}\left( \frac{1}{a_{1,1}}, ..., \frac{1}{a_{1,N-k}}\right) \frac{\prod_{i=1}^{N-k}(-1)^{N+i}\tilde{K}(-a_{1,i})}{(M+N)^{N-k}\prod_{i=1}^{N-k}a_{1,i}^{N+1}}.$$
Since $\tilde{K}(-a_{1,i})$ is different from zero for all $1\leq i\leq N-1$ and $a_{1,i}$ is different from $a_{1,j}$ for all $i\neq j$, then the matrix $M_1^k$ is invertible for all $2\leq k\leq N-1$. Similarly, for $N\leq k\leq N+2M-5$, the determinant of the matrix $M_4^k$ is given by 
$$\textup{Vandermonde}\left( \frac{1}{b_{1,M-1-q_{k}}}, ..., \frac{1}{b_{1,M-2}}\right) \frac{\prod_{i=M-1-q_{k}}^{M-2}(-1)^{i+1}b_{1,i}^{M-2+q_{k}}\tilde{U}\left(\frac{-1}{b_{1,i}}\right)}{(2M+N)^{q_k}\prod_{i=1}^{N-1}a_{1,i}^{q_k}}.$$
Also since $\tilde{U}\left(\frac{-1}{b_{1,i}}\right)$ is different from zero for all $1\leq i\leq M-2$ and $b_{1,i}$ is different from $b_{1,j}$ for all $i\neq j$, then the matrix $M_4^k$ is invertible for all $N\leq k\leq N+2M-5.$
This shows that the whole matrix $\mathcal{A}$ is invertible.
\end{proof}

\begin{remark}
The fact that the matrix $M_1^k$ is a principal minor of $M_1$ is essential for its determinant to be written under the form above. For instance, some coefficients of the last row of $M_1$ $\left(\frac{-1}{a_{1,i}^{2N-1}}\tilde{K}(-a_{1,i})-\frac{B(0)}{a_{1,i}}\right)$ may vanish.
\end{remark}

\begin{example}
For $M=N=3$, the function $f_{M,N}$ is given by 
$$f_{3,3}=\prod_{i=1}^3 \left( y+a_ix\right) \prod_{i=1}^3 \left( y+b_ix^2\right) .$$
The corresponding normal form is given by
$$N_p^{(3,3)}=xy\left( y+x^2\right) \left( y+a_{1,1}x\right) \left( y+a_{1,2}x+a_{2,2}xy\right) \left( y+b_{1,1}x^2+b_{2,1}x^3+b_{3,1}x^4+b_{4,1}x^5\right) .$$
\begin{figure}[h!]
\begin{center}
\includegraphics[width=4.5cm]{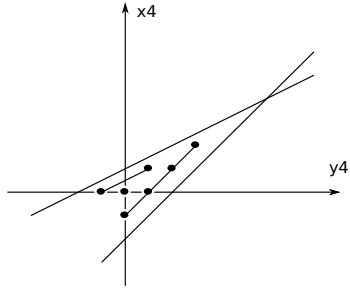}
\end{center}
\caption{The region $\mathcal{Q}^{(M,N)}$ for $M=N=3$}
\end{figure}
\end{example}

\section{The uniqueness of the normal forms.}
This section is devoted to study the uniqueness of the normal forms. From now on, we will consider $N_p$ as a notation for the normal form instead of $N_p^{(M,N)}$.\\

\noindent Let $h_{\lambda}$ be the diffeomorphism defined by: $h_{\lambda}(x,y)=(\lambda x,\lambda^2y)$. We have:
$$N_p\circ h_{\lambda}=\lambda^{2M+2N-1}N_{\lambda\cdot p} \textup{ with } \lambda\cdot p=\lambda\cdot (a_{k,i},b_{k,i})=(\lambda^{2k-3} a_{k,i},\lambda^{k-1} b_{k,i}).$$
This action of $\mathbb{C}^*$ cannot be used to "localize" the uniqueness problem as done in \cite{2} because, contrary to the quasi-homogeneous case, the topological class of the function $\frac{N_p\circ h_{\lambda}}{\lambda^{2M+2N-1}}$ jumps while $\lambda$ goes to zero. However, we are still able to prove the following:

\begin{theorem}\label{main2}
The foliations defined by $N_p$ and $N_q$, $p$ and $q$ are in $\mathcal{P}$, are equivalent if and only if there exists $\lambda$ in $\mathbb{C}^*$ such that $p=\lambda\cdot q$.
\end{theorem}

\noindent We start by the following lemma:
\begin{lemma}
Let $X$ be a germ of formal vector field given by its decomposition into the sum of its homogeneous components $X=X_{\nu_0+1}+X_{\nu_0+2}+\ldots$. If $N_p\circ e^{X_{\nu_0+1}+\ldots}=N_q$, then for all $1\leq i\leq N-1$ and $1\leq k\leq \nu_0$ we have $a_{k,i}=a'_{k,i}$ and for all $1\leq i\leq M-2$ and $1\leq k\leq \nu_1$ we have $b_{k,i}=b'_{k,i}$, where $\nu_1+1$ is the order of tangency of $\tilde{\phi}$, the lifted biholomorphism of $\phi=e^X$ by the blowing up $E_1$ defined by $E_1(x_1,y_1)=(x_1,x_1y_1).$
\end{lemma}
\begin{proof}
We consider the decomposition of the normal form into its homogeneous components:
$$N_p=N_p^{(M+N)}+N_p^{(M+N+1)}+\ldots$$
Since we have 
$$\left(e^{{X_{\nu_0+1}+\ldots}}\right)^*N_p=N_p+X_{\nu_0+1}.N_p+\ldots ,$$  
we obtain that $N_p^{(M+N+l)}=N_q^{(M+N+l)}$ for $l$ from 0 to $\nu_0-1$. The expression of $N_p^{(M+N+l)}$ only depends on the variables $a_{k,i}$ for $k\leq l+1$ and $b_{k,i}$ for $k\leq l$. 
Setting $\tilde{\phi}=e^{\tilde{X}_{\nu_1+1}+\ldots}$, the initial hypothesis leads to the following equality
$$\tilde{N}_p\circ e^{\tilde{X}_{\nu_1+1}+\ldots}=\tilde{N}_q,$$
where
$$\tilde{N}_p\left( x_1,y_1\right) =x_1y_1\left( y_1+x_1\right) \prod_{i=1}^{N-1}\left( y_1+\sum_{k=1}^{i}a_{k,i}x_1^{k-1}y_1^{k-1}\right) \prod_{i=1}^{M-2}\left( y_1+\sum_{k=1}^{N-1+2i}b_{k,i}x_1^{k}\right) .$$
Similarly we obtain $\tilde{N}_p^{(M+1+l)}=\tilde{N}_q^{(M+1+l)}$ for $l$ from 0 to $\nu_1-1$. The expression of $\tilde{N}_p^{(M+1+l)}$ only depends on the variables $a_{k,i}$ for $k\leq l$ (except for $l=0$ as $\tilde{N}_p^{(M+1)}$ depends on $a_{1,i}$) and $b_{k,i}$ for $k\leq l+1$. Now, we claim that for all $l$ from 0 to $\nu_0-1$, 
$$N_p^{(M+N+l)}=N_q^{(M+N+l)} \textup{ and } \tilde{N}_p^{(M+1+l)}=\tilde{N}_q^{(M+1+l)} \Leftrightarrow a_{k,i}=a'_{k,i} \textup{ and } b_{k,i}=b'_{k,i} \forall k\leq l+1.$$ 
This fact can be proved by induction on $l\leq \nu_0-1$. For $l=0$, we have the following two equalities 
$$N_p^{(M+N)}=N_q^{(M+N)} \textup{ and }  \tilde{N}_p^{(M+1)}=\tilde{N}_q^{(M+1)}.$$
Since the conjugacy preserves a fixed numbering of the branches, we obtain that $a_{1,i}=a'_{1,i}$ and $b_{1,i}=b'_{1,i}$.
Suppose that $a_{k,i}=a'_{k,i}$ and $b_{k,i}=b'_{k,i}$ for $l<\nu_0-1$. Then we have $N_p^{(M+N+l)}=N_q^{(M+N+l)}$ with 
$$N_p^{(M+N+l)}=\sum_{i=1}^{N-1}a_{l+1,i}xy^l\frac{N_p^{(M+N)}}{y+a_{1,i}x}+\sum_{i=1}^{M-2}b_{l,i}x^{l+1}\frac{N_p^{(M+N)}}{y}+H_{a,b}(x,y),$$
where $H_{a,b}$ is a function which depends on $a_{k,i}$ for $k<l+1$ and $b_{k,i}$ for $k<l$. This implies that $a_{l+1,i}=a'_{l+1,i}$. Similarly, we have $\tilde{N}_p^{(M+1+l)}=\tilde{N}_q^{(M+1+l)}$ with
$$\tilde{N}_p^{(M+1+l)}=\sum_{i=1}^{N-1}a_{\frac{l}{2}+1,i}x_1^{\frac{l}{2}}y_1^{\frac{l}{2}}\frac{\tilde{N}_p^{(M+1)}}{a_{1,i}}+\sum_{i=1}^{M-2}b_{l+1,i}x_1^{l+1}\frac{\tilde{N}_p^{(M+1)}}{y_1+b_{1,i}x_1}+\tilde{H}_{a,b}(x_1,y_1)$$
where the first term exists only if $l$ is even and greater than or equal to two and $\tilde{H}_{a,b}(x_1,y_1)$ is a function which depends on $a_{k,i}$ for $k<l$ and $b_{k,i}$ for $k<l+1$. This implies that $b_{l+1,i}=b'_{l+1,i}$.\\
Now, we know that $\nu_0\leq \nu_1$. So we claim that for all $\nu_0\leq l\leq \nu_1-1$, 
$$\tilde{N}_p^{(M+1+l)}=\tilde{N}_q^{(M+1+l)} \Longleftrightarrow b_{k,i}=b'_{k,i} \hspace{0.5cm}\forall k\leq l+1.$$
For $l=\nu_0$, we know that $a_{\nu_0,i}=a'_{\nu_0,i}$ and $b_{\nu_0,i}=b'_{\nu_0,i}$. Similarly we obtain that $b_{\nu_0+1,i}=b'_{\nu_0+1,i}$. 
Suppose that $b_{k,i}=b'_{k,i}$ for $l<\nu_1-1$. Then we have $\tilde{N}_p^{(M+1+l)}=\tilde{N}_p^{(M+l+l)}$ where 
\[\begin{array}{rl}
\tilde{N}_p^{(M+1+l)}= & \displaystyle\sum_{\substack{i=1}}^{N-1}a_{\frac{l}{2}+1,i}x_1^{\frac{l}{2}}y_1^{\frac{l}{2}}\frac{\tilde{N}_p^{(M+1)}}{a_{1,i}}+\displaystyle\sum_{\substack{i=1}}^{M-2}b_{l+1,i}x_1^{l+1}\frac{\tilde{N}_p^{(M+1)}}{y_1+b_{1,i}x_1}+\\
 & \displaystyle\sum_{\substack{i=1}}^{N-1}\displaystyle\sum_{\substack{i=1}}^{M-2}\displaystyle\sum_{\substack{2k_1+k_2=l+3\\k_1,k_2\neq 1}}a_{k_1,i}b_{k_2,j}x_1^{k_1+k_2-1}y_1^{k_1-1}\frac{\tilde{N}_p^{M+1}}{a_{1,i}(y_1+b_{1,j}x_1)}+\tilde{H}_{a,b}(x_1,y_1).
\end{array}\]
To show that $b_{l+1,i}=b'_{l+1,i}$, it is enough to show that $k_1<\nu_0+1$. In fact, by definition we have $k_1=\frac{l+3-k_2}{2}$. So, using that $l\leq \nu_1-1$, $k_2>1$ and that $\nu_1\leq 2\nu_0$, we conclude that $k_1<\nu_0+1$.
\end{proof}
\noindent A process of blowing-up $E$ is said to be a \textit{chain process} if, either $E$ is the standard blowing-up of the origin of $\mathbb{C}^2$, or $E=E'\circ E''$ where $E'$ is a chain process and $E''$ is the standard blowing-up of the of a point that belongs to the smooth part of the highest irreducible component of $E'$. The length of a chain process of blowing-up is the total number of blowing-up and the height of an irreducible component $D$ of the exceptional divisor of $E$ is the minimal number of blown-up points so that $D$ appears. A chain process of blowing-up admits \textit{privileged systems of coordinates} $(x,y)$ in a neighborhood of the component of maximal height such that $E$ is written
$$E:(x,t)\longmapsto (x,tx^h+t_{h-1}x^{h-1}+t_{h-2}x^{h-2}+\ldots +t_1x).$$
The values $t_i$ are the positions of the successive centers in the successive privileged coordinates and $x=0$ is a local equation of the divisor.\\
\noindent Let $\phi$ be a germ of biholomorphism tangent to the identity map at order $\nu_0+1\geq 2$ and fixing the curves $\{x=0\}$ and $\{y=0\}$. The function $\phi$ is written 
\begin{equation}
(x,y)\longmapsto \big(x(1+A_{\nu_0}(x,y)+\ldots),y(1+B_{\nu_0}(x,y)+\ldots)\big)
\end{equation}
where $A_{\nu_0}$ and $B_{\nu_0}$ are homogeneous polynomials of degree $\nu_0$. The following lemma can be proved by induction on the height of the component:

\begin{lemma}
The biholomorphism $\phi$ can be lifted-up through any chain process $E$ of blowing-up with length smaller $\nu_0+1$: there exists $\tilde{\phi}$ such that $E\circ\tilde{\phi}=\phi\circ E$. The action of $\tilde{\phi}$ on any component of the divisor of height less than $\nu_0$ is trivial. Its action on any component of height $\nu_0+1$ is written in privileged coordinates 
$$(0,t)\longmapsto \big(0,t+t_1B_{\nu_0}(1,t_1)-t_1A_{\nu_0}(1,t_1)\big)$$
where $t_1$ is the coordinate of the blown-up point on the first component of the irreducible divisor.
\end{lemma} 

\begin{definition}
A germ of biholomorphism $\phi$ is said is said to be \textit{dicritical} if $\phi$ written 
$$(x,y)\longmapsto \big(x+A_{\nu}(x,y)+\ldots,y+B_{\nu}(x,y)+\ldots\big),$$
$xB_{\nu}(x,y)-yA_{\nu}(x,y)$ vanishes.
\end{definition}

\noindent We can now prove the main Theorem \ref{main2} of this section.\\

\noindent\textit{Proof of Theorem \ref{main2}.}
Suppose that there exists a conjugacy relation 
\begin{equation}
N_p\circ\phi=\psi\circ N_q.
\end{equation}
Following \cite{4}, we can suppose that $\psi$ is a homothety $\gamma$Id. The biholomorphism $\phi$ can be supposed tangent to the identity. In fact, since $\phi$ lets the curves $\{x=0\}$, $\{y=0\}$ and $\{y+x^2=0\}$ invariant, then it can be written
$$(x,y)\longmapsto \left(\lambda x(1+A_{\nu_0}(x,y)+\ldots),\lambda^2 y(1+B_{\nu_0}(x,y)+\ldots)\right),$$
for some $\lambda\neq 0$. Then
$$N_p\circ\phi\circ h_{\lambda}^{-1}=\gamma N_q\circ h_{\lambda}^{-1}=cN_{\lambda^{-1}\cdot q},$$
where $c$ stands for some non vanishing number. Since $\phi\circ h_{\lambda}^{-1}$ is tangent to the identity, we find that $c=1$. Thus, setting for the sake of simplicity $q=\lambda^{-1}\cdot q$ and $\phi=\phi\circ h_{\lambda}^{-1}$, we are led to the relation
$$ N_p\circ\phi=N_q,$$
where $\phi$ can be written under the form $(11)$.

\noindent The proof reduces to show that in this situation, we have $p=q$. Using Lemma (3.1), we know that for all $1\leq i\leq N-1$ and $1\leq k\leq \nu_0$ we have $a_{k,i}=a'_{k,i}$ and for all $1\leq i\leq M-2$ and $1\leq k\leq \nu_1$ we have $b_{k,i}=b'_{k,i}$. This means that, based on the structure of the normal form, to show that for any $k\leq N-1$, $a_{k,i}=a'_{k,i}$, it is enough to show that $\nu_0\geq N-1$. In the same way, to show that for any $k\leq 2M-N-5$, $b_{k,i}=b'_{k,i}$, it is enough to show that $\nu_1\geq N+2M-5$. Thus, the proof results from the following proposition:

\begin{proposition}
If $N_p\circ \phi=N_q$, then the following assertions hold:
\begin{enumerate}
\item If $\phi$ is dicritical then $p=q$.
\item If $\phi$ is non-dicritical then $\nu_0\geq N$.
\item If $\phi$ and $\tilde{\phi}$ are non-dicritical then $\nu_1\geq 2M+N-5$.
\item If $\tilde{\phi}$ is dicritical then $p=q$.
\end{enumerate}
\end{proposition}

\begin{proof}
\begin{enumerate}

\item If $\nu_0\geq 2M+N-5$ then $\nu_1\geq 2M+N-5$ and $\nu_0\geq N-1$. So, by Lemma (3.1), we have $p=q$. Suppose that $\nu_0<2M+N-5$. Since $\phi$ is tangent to the identity, then it is the time one of the flow of a formal dicritical vector field
$$\phi=e^{\hat{X}}.$$
Its homogeneous part of degree $\nu_0+1$ is radial and is written $\phi_{\nu_0}R$ where $\phi_{\nu_0}$ stands for a homogeneous polynomial function of degree $\nu_0$ and $R$ for the radial vector field $x\partial_{x}+y\partial_{y}$. The initial hypothesis can be expressed as follows
$$\left(e^{\hat{X}}\right)^*N_p=N_p+\phi_{\nu_0}R.N_p+\ldots=N_q.$$
In this relation, the valuation of $\phi_{\nu_0}R.N_p$ is at least $\nu_0+M+N$. Lemma (3.1) implies that the first non-trivial homogeneous part of the previous relation is of valuation $\nu_0+M+N$ and it is written 

$$N_p^{(\nu_0+M+N)}+\phi_{\nu_0}R.N_p^{(M+N)}=N_q^{(\nu_0+M+N)}.$$
Since $N_p^{(M+N)}$ is homogeneous, then this relation becomes 
$$N_p^{(\nu_0+M+N)}-N_q^{(\nu_0+M+N)}+(M+N)\phi_{\nu_0}N_p^{(M+N)}=0.$$
The homogeneous component of degree $\nu_0+M+N$ in $N_p$ is written 
\begin{itemize}
\item If $\nu_0+1\leq N-1$, then 
$$N_p^{(\nu_0+M+N)}= \displaystyle\sum_{i=1}^{N-1}a_{\nu_0+1,i}xy^{\nu_0}\frac{N_p^{(M+N)}}{y+a_{1,i}x}+\displaystyle\sum_{i=1}^{M-2}b_{\nu_0,i}x^{\nu_0+1}\frac{N_p^{(M+N)}}{y}+ H_{a,b}(x,y)$$
where $H_{a,b}$ is a function which depends on $a_{k,i}$ for $k<\nu_0+1$ and $b_{k,i}$ for $k<\nu_0$. Since $a_{1,i}=a'_{1,i}$ and $b_{\nu_0,i}=b'_{\nu_0,i}$, then the difference $N_p^{(\nu_0+M+N)}-N_q^{(\nu_0+M+N)}$ is written 
$$N_p^{(M+N)}\left(\sum_{i=1}^{N-1}\frac{\lambda_ixy^{\nu_0}}{y+a_{1,i}x}\right),$$
where $\lambda_i=a_{\nu_0+1,i}-a'_{\nu_0+1,i}$. Therefore, the polynomial function $\phi_{\nu_0}$ must coincide with 
$$-\frac{1}{M+N}\sum_{i=1}^{N-1}\frac{\lambda_ixy^{\nu_0}}{y+a_{1,i}x}$$
which happens to be polynomial if and only if $\lambda_i$ vanishes for all $i$ and therefore $\phi_{\nu_0}$ must be the zero polynomial.
\item If $\nu_0+1> N-1$, then 
$$N_p^{(\nu_0+M+N)}=\displaystyle\sum_{i=1}^{M-2}b_{\nu_0,i}x^{\nu_0+1}\frac{N_p^{(M+N)}}{y}+ H_{a,b}(x,y)$$
where $H_{a,b}$ is a function which depends on $b_{k,i}$ for $k<\nu_0$. Since $b_{\nu_0,i}=b'_{\nu_0,i}$, then the difference $N_p^{(\nu_0+M+N)}-N_p^{(\nu_0+M+N)}$ is zero. As a consequence $\phi_{\nu_0}$ must be the zero polynomial.
\end{itemize}
\item We suppose that $\nu_0<N$. We know that $\phi$ can be written as follows 
$$(x,y)\longmapsto \left(x(1+A_{\nu_0}(x,y)+\ldots),y(1+B_{\nu_0}(x,y)+\ldots)\right).$$
Since the action of $\phi$ on any component of height $\nu_0+1$ conjugates the complete cones, then the function $tB_{\nu_0}(1,t)-tA_{\nu_0}(1,t)$ vanishes on $\{0,\infty,a_{1,1},\ldots, a_{1,\nu_0}\}$, which is the common tangent cone of $N_p$ and $N_q$. Since the degree of $tB_{\nu_0}(1,t)-tA_{\nu_0}(1,t)$ is at most $\nu_0+1$, then it is the zero polynomial. Hence,
$$xyB_{\nu_0}(x,y)-xyA_{\nu_0}(x,y)=0,$$
which is impossible since $\phi$ is non-dicritical.
\item Suppose that $\nu_1<2M+N-5$. The functions $A_{\nu_0}$ and $B_{\nu_0}$ are homogeneous of degree $\nu_0$. So, we write them as
$$A_{\nu_0}(x,y)=\sum_{i+j=\nu_0}\alpha_{i,j}x^iy^j \textup{ and } B_{\nu_0}(x,y)=\sum_{i+j=\nu_0}\beta_{i,j}x^iy^j.$$
Since $\nu_0\geq N$, then the function $f(t)$, defined by 
$$f(t)=tB_{\nu_0}(1,t)-tA_{\nu_0}(1,t)=t\sum_{i+j=\nu_0}(\beta_{i,j}-\alpha_{i,j})t^j,$$
vanishes at $\{0,\infty,a_{1,1},\ldots, a_{1,N-1}\}$.
The biholomorphism $\tilde{\phi}$ is given by $\tilde{\phi}=E_1^{-1}\circ\phi\circ E_1$. So, it can be written as
$$\tilde{\phi}(x_1,y_1)=\left(x_1(1+A(x_1, x_1y_1)),y_1(1+B(x_1,x_1y_1)-A(x_1,x_1y_1)+\ldots)\right),$$
where the lifted homogeneous parts of degree $\nu_0$ of $A$ and $B$ has the form 
$$A_{\nu_0}(x_1,x_1y_1)=\sum_{i+j=\nu_0}\alpha_{i,j}x_1^{\nu_0}y_1^j \textup{ and } B_{\nu_0}(x_1,x_1y_1)=\sum_{i+j=\nu_0}\beta_{i,j}x_1^{\nu_0}y_1^j.$$
Since the order of tangency of $\phi$ is $\nu_0+1$ then there exists $i$ and $j$ satisfying $i+j=\nu_0$ such that $\alpha_{i,j}\neq 0$ or $\beta_{i,j}\neq 0$. Let $j_0$ be the smallest such $j$. So, we have
$$\tilde{\phi}(x_1,y_1)=\left(x_1(1+\tilde{A}_{\nu_1}(x_1,y_1)+\ldots),y_1(1+\tilde{B}_{\nu_1}(x_1,y_1)+\ldots)\right),$$
where 
$$\tilde{A}_{\nu_1}(x_1,y_1)=\sum_{i+j=\alpha_0\leq j_0}\alpha'_{i,j}x_1^{\nu_0+i}y_1^j \textup{ and }  \tilde{B}_{\nu_1}(x_1,y_1)=\sum_{i+j=\alpha_0\leq j_0}(\beta'_{i,j}-\alpha'_{i,j})x_1^{\nu_0+i}y_1^j.$$
So, the order of tangency of $\tilde{\phi}$, $\nu_1+1$, is equal to $\nu_0+\alpha_0+1$. We define the function $\tilde{f}$ by 
$$\tilde{f}(t)=t\tilde{B}_{\nu_1}(1,t)-t\tilde{A}_{\nu_1}(1,t).$$
We know that $\nu_1\geq\nu_0\geq N$. Since the action of $\tilde{\phi}$ on any component of height $\nu_1+1$ conjugates the complete cones, then, if $\nu_1=N$, the function $\tilde{f}$ vanishes at 0, 1 and $\infty$. Since $\tilde{\phi}$ is non-dicritical then $\alpha_0+1$ must be greater than or equal to 3. This implies that $j_0\geq 2$ and so for all $j<2$ satisfying $i+j=\nu_0$, we have $\alpha_{i,j}=\beta_{i,j}=0$. However, the function $f(t)=t^3\sum_{i+j=\nu_0}(\beta_{i,j}-\alpha_{i,j})t^{j-2}$ vanishes at $\{0,\infty,a_{1,1},\ldots, a_{1,N-1}\}$. Since $\phi$ is non-dicritical, then $\nu_0-2$ must be greater than or equal to $N$. This implies that $\nu_1$ must be at least $N+4$ which is impossible. Thus, $\nu_1$ must be greater than $N$. We proceed similarly at each level. Finally, if $\nu_1=2M+N-6$, then the function $\tilde{f}$ vanishes at $\{0,1,\infty, b_{1,1}, \ldots, b_{1,M-3}\}$. Since $\tilde{\phi}$ is non-dicritical, then $\alpha_0+1$ must be at least $M$. This implies that $j_0\geq M-1$. Similarly, we must have $\nu_0-M+1\geq N$. As a consequence, $\nu_1$ must be at least $2M+N-2$ which is impossible.
\item The proof is similar to that of the first point, noting that we necessarily have $a_{k,i}=a'_{k,i}$ for all $1\leq i\leq N-1$ and $1\leq k\leq\nu_0$. 
\end{enumerate}
\end{proof}

\Addresses

\end{document}